\documentclass[12pt]{amsart}
\usepackage{amsmath}
\usepackage{amssymb}
\usepackage{amsfonts}
\usepackage{amsthm}
\usepackage{xcolor}
\usepackage{enumerate}
\usepackage{hyperref}
\hypersetup{
    bookmarks=true,         % show bookmarks bar?
    colorlinks=true,       % false: boxed links; true: colored links
    linkcolor=blue,          % color of internal links (change box color with linkbordercolor)
    citecolor=red!50!black,        % color of links to bibliography
    filecolor=cyan,         % color of file links
    urlcolor=magenta        % color of external links
}

\newtheorem{theo}{Theorem}[section]
\newtheorem*{theo*}{Theorem}

\newtheorem{lemm}[theo]{Lemma}

\newtheorem{exam}[theo]{Example}

\newcommand{\N}{\mathbb{N}}

\title[On Cohen-Ramanujan expansions]{A reciprocity theorem for the Cohen-Ramanujan sums and its application to Cohen-Ramanujan expansions in the second variable}
\begin{document}

\author[K V Namboothiri]{K Vishnu Namboothiri}
\address{Department of Mathematics, Baby John Memorial Government College, Chavara, Sankaramangalam, Kollam, Kerala - 691583, India}
\address{Department of Collegiate Education, Government of Kerala, India}
\email{kvnamboothiri@gmail.com}

\author{Vinod Sivadasan}
 \address{Department of Mathematics, University College, Thiruvananthapuram (Research Centre of the University of Kerala), Kerala - 695034, India}
 \address{Department of Mathematics, College of Engineering Trivandrum, Thiruvananthapuram, Kerala - 695034, India}
 \address{Department of Collegiate Education, Government of Kerala, India}
 \email{wenods@gmail.com}

%\footnotetext[1]{\textit{first author}}
%\footnotetext[2]{\textit{corresponding author}}

 \begin{abstract}
For an arithmetical function $f$, its Ramanujan expansion  is a series expansion in the form $f(n)=\sum\limits_{k=1}^{\infty}a(k) c_k(n)$ where  $a(k)$ are complex numbers  and $c_k(n):= \sum\limits_{\substack{m=1\\(m, k)=1}}^{k}e^{\frac{2\pi imn}{k}}$ is the Ramanujan sum. 
%For certain functions $f$, K. R. Johnson proved that the existence of a series expansion in this form is necessary and sufficient for the existence of series expansions in the form $\sum\limits_{k=1}^{\infty}b(k/n) c_n(k)$.   
Here we prove a reciprocity result on Cohen-Ramanujan sums $c_k^s(n) :=\sum\limits_{\substack{{h=1}\\(h,k^s)_s=1}}^{k^s}e^{\frac{2\pi i n h}{k^s}}$ to change the position of $k$ and $n$ in a twisted function and use it to prove that  for certain arithmetical functions $f$,  Cohen-Ramanujan series expansions in the form  $\sum\limits_{k=1}^{\infty}a(k) c_k^{(s)}(n)$ exist if and only if expansions in the form  $\sum\limits_{k=1}^{\infty}b(k/n) c_n^{(s)}(k)$ exist.
 \end{abstract}
 
 \maketitle
\section{Introduction}
For an arithmetical function $f$, its \emph{Ramanujan expansion} is a series expansion in the form $f(n)=\sum\limits_{k=1}^{\infty}a(k) c_k(n)$ where $a(k)$ are complex numbers and  $c_k(n)$ is the \emph{Ramanujan sum} defined as
\begin{align*}
c_k(n) :=\sum\limits_{\substack{{m=1}\\(m,k)=1}}^{k}e^{\frac{2 \pi imn}{k}}.
\end{align*} provided that the series on the right of $f(n)$ converges. Srinivasa Ramanujan gave such expansions for some arithmetical functions in \cite{ramanujan1918certain} which were pointwise convergent. Such expansions and necessary and sufficient conditions for such expansions to exist became a topic of discussion in many papers. A. Wintner \cite{wintner1943eratosthenian} gave a condition for the existence of Ramanujan expansions which was later improved by H. Delange in   \cite{delange1976ramanujan}. L. Lucht gave some other conditions in \cite{lucht1995ramanujan} and \cite{lucht2010survey} for such expansions to exist. The paper \cite{murty2013ramanujan} by Ram Murty discusses some new and old results related to these expansions. 

Ramanujan sum was generalized by many later. E.\ Cohen gave a generalization of the Ramanujan sum in \cite{cohen1949extension} defining
\begin{align}\label{cohen-ram-sum}
c_k^s(n) :=\sum\limits_{\substack{{h=1}\\(h,k^s)_s=1}}^{k^s}e^{\frac{2\pi i n h}{k^s}},
\end{align} where $(a,b)_s$ is the generalized gcd of $a$ and $b$. A. Chandran and K. V. Namboothiri derived some conditions for the existence of Ramanujan series like expansions for certain arithmetical functions in terms of the generalized Ramanujan sum (\ref{cohen-ram-sum}) (henceforth called as the  \emph{Cohen-Ramanujan sum})  in \cite{chandran2023ramanujan}. The method of arguments they used were those employed by L. Lucht in \cite{lucht1995ramanujan} and \cite{lucht2010survey}.

 Let us call $k$ to be the first variable and $n$ the second in $c_k(n)$ and $c_k^s(n)$. The Ramanujan expansions and Cohen-Ramanujan expansions mentioned in the papers above are all in the first variable $k$ of the (Cohen-)Ramanujan sum. K. R. Johnson took a different route and proved that for certain arithmetical functions $f$ having Ramanujan expansions (which he called as the $C$-series expansions), there exist expansions in the second variable (\cite[Theorem 7]{johnson1982reciprocity}). Precisely, he proved that 
 
 \emph{A function f has an absolutely convergent C-series representation $f(n) = \sum\limits_{k=1}^{\infty} a(k)c_k(n)$ with $a(k)=0$ if $k$ is not square-free, if and only if  there exists an absolutely convergent $C'$-series
representation $f(n) n^* \mu(\overline{n}) = \sum\limits_{k=1}^{\infty} b(k/n^*)c_n (k)$.}
 
 Here $\overline{n}$ is the largest square-free divisor of $n$, called the \emph{core} of $n$ and $n^*=n/\overline{n}$.
 He proved it with the help of a \emph{reciprocity result}  (\cite[Theorem 6]{johnson1982reciprocity}) which enabled him to interchange the positions of $k$ and $n$ in the Ramanujan sum twisted by the M\"{o}bius function $\mu$. Here we derive an analogous reciprocity result (Theorem \ref{th:mu-c-reciprocity}) involving Cohen-Ramanujan sums and consequently show that for certain arithmetical functions $f$, an expansion
in the form $\sum\limits_{k=1}^{\infty} a(k)c_k^{(s)}(n^s)$ exists if and only if an expansion in the form  $\frac{\mu(\overline{n})}{(n^*)^s}\sum\limits_{k=1}^{\infty}\xi_{n^*}\,(k)b(k/n^*) c_n^{(s)}(k^s)$ exists (Theorem \ref{th:CR-exp-in-other-var}). Here  $\xi_d(k)=d$ if $d|k$ and 0 otherwise.

\section{Notations and basic results}

By $(a, b)_s$, we mean the \emph{generalized GCD} of $a$ and $b$ defined to be the largest $d^s \in \mathbb{N}$ ($d\in \N$)  such that $d^s|a$ and $d^s|b$. Hence $(a,b)_1$ is the usual GCD $(a,b)$. By $J_s$ we mean the \emph{Jordan totient function} defined as $J_s(n) := n^s \prod\limits_{\substack{{p \mid n}}}(1-\frac{1}{p^s})$. Note that $J_1(n)=\varphi(n)$ where $\varphi$ is the Euler totient function. $\omega(k)$ will be used to denote the number of distinct prime divisors of $k$. Unless otherwise stated, $p$ will be used to represent primes.

\emph{The Cohen-Ramanujan sum} (or \emph{generalized Ramanujan sum}) defined by E. Cohen in \cite{cohen1949extension} is the sum given in equation (\ref{cohen-ram-sum}).
By (\cite[Theorem 1]{cohen1949extension}), $c_k^{(s)}(n)$ is multiplicative in $k$. When $k$ is a prime power, the values of the Cohen-Ramanujan sum are as folows (\cite[Theorem 3]{cohen1949extension}).

 \begin{align}
  c_{p^j}^{(s)}(n) = \begin{cases}
                         p^{sj}-p^{s(j-1)} &\quad \text{if }p^{sj}|n\\
                         -p^{s(j-1)} &\quad \text{if }p^{s(j-1)}||n\\
                         0&\quad  \text{if }p^{s(j-1)}\nmid n.
                        \end{cases}\label{eq:crs-primes}
 \end{align}
  Cohen-Ramanujan sum satisfies the following relation with the M\"{o}bius function $\mu$ (\cite[Theorem 3]{cohen1949extension}):
\begin{align}
                    c_k^{(s)}(n)=\sum\limits_{\substack{d|k\\d^s|n}}\mu(k/d)d^s\label{eq:crs-mu}.
                   \end{align}
 Also, with the Jordan totient function, it holds the relation (\cite[Theorem 2]{cohen1959trigonometric})
 \begin{align}
       c_k^{(s)}(n) = \frac{J_s(n)\mu(m)}{J_s(m)}\quad\text{ where } m=\frac{n}{(k,n)}\label{eq:crs-jordan}.
      \end{align}
Sometimes we club the conditions $d|k$ and $d^s|n$ into the single condition $d^s|(n, k^s)_s$.

By the notation $d^m||n$ we mean that $d^m|n$ and $d^{m+1}\nmid n$.
For a prime $p$, the exponent of $p$ in the prime factorization of $k$  will be denoted by $e_p(k)$. Hence $p^{e_p(k)}||k$. Analogously, we write $a := e_p^{(s)}(n)$ if $p^{as}|n$ and $p^{(a+1)s}\nmid n$. In this case, we also use the notation $p^{as}||_s n$.
We use the M\"{o}bius inversion formula and a variant of it many times in this paper. For two arithmetical functions $f$ and $g$, it states that 
\begin{align}
g(k) = (\mu*f)(k) :=\sum\limits_{d|k}\mu(d)f\left(\frac{k}{d}\right) \Longleftrightarrow f(k) = (g*u)(k) = \sum\limits_{d|k} g(d) 
\end{align} where $*$  is the Dirichlet product, and $u\equiv 1$ is the unit function defined by $u(n)=1$ for all $n\in \N$.

A more general inversion formula was given by Hardy and Wright (\cite[Theorem 270]{hardy1979introduction}) which states that 
$g(k)=\sum\limits_{\substack{d=1}}^{\infty}\mu(d)f(kd)$ if and only if $f(k) = \sum\limits_{\substack{d=1}}^{\infty}g(kd)$.

\section{Main Results and proofs}
We have already stated in the previous section that  Cohen-Ramanujan sum $c_k^{(s)}(n)$ is multiplicative in $k$. However, it is not multiplicative in $n$. It was proved by Donovan and Rearick in  \cite{donovan1966ramanujan} that though $c_k(n)$ is not multiplicative in $n$, the twisted function $\mu(k)c_k(n)$ is multiplicative in $n$ when $k$ is sqaure-free. We will prove an analogous result here for the Cohen-Ramanujan sum.
\begin{theo}\label{th:mu-ck-mult}
 For positive integers $m,n$ with $(m,n)=1$, and a square-free positive integer $k$,  we have $\mu(k)c_k^{(s)}(mn) = \mu(k)c_k^{(s)}(m) \times \mu(k)c_k^{(s)}(n)$.
\end{theo}

\begin{proof}
 Write $k = k_1k_2$ with $(k_1,n)=(k_2,m)=1$. Then
 \begin{align*}
  \mu(k_1k_2) c_{k_1k_2}^{(s)}(mn)=\mu(k_1) c_{k_1}^{(s)}(mn) \times \mu(k_2) c_{k_2}^{(s)}(mn).
 \end{align*}
Now use identity (\ref{eq:crs-mu}) to proceed further to get the above to be equal to
\begin{align*}
 \mu(k_1)\sum\limits_{\substack{d|k_1\\d^s|mn}}d^s\mu(k_1/d)\times \mu(k_2)\sum\limits_{\substack{d|k_2\\d^s|mn}}d^s\mu(k_2/d).
\end{align*}
By our assumption on $k_1, k_2, m$, and $n$, we can see that $\{d:d|k_1\text{ and }d^s|mn\}=\{d:d|k_1\text{ and }d^s|m\} = \{d:d|k\text{ and }d^s|m\}$. Similarly,  $\{d:d|k_2\text{ and }d^s|mn\}=\{d:d|k\text{ and }d^s|n\}$. So

\begin{align*}
  \mu(k_1k_2) c_{k_1k_2}^{(s)}(mn) &= \mu(k_2)\sum\limits_{\substack{d|k\\d^s|m}}d^s\mu(k_1/d)\times \mu(k_1)\sum\limits_{\substack{d|k\\d^s|n}}d^s\mu(k_2/d)\\
  &= \sum\limits_{\substack{d|k\\d^s|m}}d^s\mu(k/d) \times  \sum\limits_{\substack{d|k\\d^s|n}}d^s\mu(k/d)\\
  &= c_k^{(s)}(m) c_k^{(s)}(n)\\
  &= \mu(k) c_k^{(s)}(m) \times \mu(k) c_k^{(s)}(n).
  \end{align*}

\end{proof}

K. R. Johnson  used the the multiplicative property of $\mu(k)c_k(.)$ to show that $\mu(k)c_k(n)=\sum\limits_{d|(k,n)} d\mu(d)$ for any square-free $k$ (\cite[Lemma 1]{johnson1982reciprocity}). He did not give  details for proving it, but just stated that the multiplicative nature of both sides of this equality can be used to prove this quickly. We will derive an analogous result giving a detailed proof.

\begin{theo}\label{th:mu-ck-sum}
For  square-free positive integer $k$, we have
\begin{align*}
 \mu(k)c^{(s)}_k(n)=\sum\limits_{d^s|(k^s,n)_s} d^s\mu(d)
\end{align*}
\end{theo}
\begin{proof}
 From Theorem \ref{th:mu-ck-mult}, $\mu(k)c^{(s)}_k(n)$ is multiplicative in $n$. So if we show that $\sum\limits_{d^s|(k^s,n)_n} d^s\mu(d)$ is also multiplicative in $n$, we can complete the proof by verifying the identity in the statement for prime powers.

 Let $m,n$ be positive integers such that $(m,n)=1$. Now $d^s|(k^s, mn)_s$ implies that $d^s|k^s$ and $d^s|mn$. Since $(m,n)=1$,  it must be true that $d=d_1d_2$ with $(d_1,d_2)=1$ and $(d_1,n)=1=(d_2,m)$. So as we saw previously, $d^s|(k^s, mn)_s$ if and only if $d_1^s|k^s, d_1^s|m, d_2^s|k^s$, and $d_2^s|n$.  Hence \begin{align*}
\sum\limits_{d^s|(k^s,n)_n} d^s\mu(d)=\sum\limits_{d_1^s|(k^s,m)_s} d_1^s\mu(d_1)\times \sum\limits_{d_2^s|(k^s,n)_s} d_2^s\mu(d_2).
\end{align*}

Now we will show that for every prime $p$ and  positive integer $a$, we have $\mu(k)c^{(s)}_k(p^a)=\sum\limits_{d^s|(k^s,p^a)_s} d^s\mu(d)$. If $p|k$, then since $k$ is square-free, $p^s||_s k^s$. In addition, if $a\geq s$, then $d^s|(k^s,p^a)_s\Rightarrow d^s = 1$ or $d^s=p^s$. Hence if $p|k$ and $a\geq s$, then $\sum\limits_{d^s|(k^s,p^a)_s} d^s\mu(d)=\mu(1)+p^s\mu(p) = 1-p^s.$ Now $\mu(k)c^{(s)}_k(p^a) = \mu(p)c^{(s)}_p(p^a) \times \mu(k/p)c^{(s)}_{k/p}(p^a)$. From equation  \ref{eq:crs-primes}, $c_p^{(s)}(p^a) = p^s-1$ and $c^{(s)}_{k/p}(p^a)=\prod\limits_{\substack{i = 1\\p_i|k,\,p_i\neq p}}^r  c^{(s)}_{p_i}(p^a) = \prod\limits_{\substack{i = 1\\p_i\neq p}}^r (-p^{s(1-1)})=(-1)^{\omega(k/p)}$. Also, since $k$ is square-free, $\mu(k/p) = (-1)^{\omega(k/p)}$. Therefore $\mu(k)c^{(s)}_k(p^a)=(-1)(p^s -1)=1-p^s$.

Now if $p\nmid k$ or $a<s$, then $d^s|(k^s,p^a)_s\Rightarrow d^s=1$. So $\sum\limits_{d^s|(k^s,p^a)_s} d^s\mu(d)=1$. Now $\mu(k)c^{(s)}_k(p^a) = \prod\limits_{p_i|k}\mu(p_i) c^{(s)}_{p_i}(p^a) = \prod\limits_{p_i|k} (-1)(-1)=1$. So in all the cases, we get $\mu(k)c^{(s)}_k(p^a)=\sum\limits_{d^s|(k^s,p^a)_s} d^s\mu(d)$. This completes the proof of the theorem.
\end{proof}

Recall that for a positive integer $k$,  the \emph{core} of $k$ is defined to be its largest square-free divisor denoted by $\overline{k}$ and $k^* = k/\overline{k}$. Hardy proved in \cite{hardy1921note} that $c_k(n)=0$ if $k^*\nmid n$. This is equivalent to the statement that $c_{p^r}(n)=0$ unless $p^{r-1}|n$. An analogous statement in terms of the Cohen-Ramanujan sum is appearing in equation  (\ref{eq:crs-primes}). We use is to reformulate Hardy's result as follows and skip the proof of it as it follows easily from equation  (\ref{eq:crs-primes}).
\begin{lemm}
$c_k^{(s)}(n)=0$ if $(k^*)^s\nmid n$.
\end{lemm}

In \cite[Lemma 2]{johnson1982reciprocity}, K. R. Johnson stated that $c_k(nk^*) = k^*c_{\overline{k}}(n)$. He did not provide any details on how to arrive at this result, but just mentioned that one may use Hardy's result  (which is the above lemma with $s=1$) to verify it. We give a generalization of this result in terms of the Cohen-Ramanujan sum and prove it.
\begin{lemm}\label{th:ck-ck-core}
 $ c_k^{(s)}(n(k^*)^s)=(k^*)^s c_{\overline{k}}^{(s)}(n).$
 \end{lemm}
 \begin{proof}
\begin{align*}
  c_k^{(s)}(n(k^*)^s) &=\prod\limits_{p^r||k} c_{p^r}^{(s)}(n(k^*)^s)\\
  &=\prod\limits_{\substack{p^r||k\\p^s|n}} c_{p^r}^{(s)}(n(k^*)^s)\times \prod\limits_{\substack{p^r||k\\p^s\nmid n}} c_{p^r}^{(s)}(n(k^*)^s)\\
  &= \prod\limits_{\substack{p^r||k\\p^s| n}} (p^{rs}-p^{(r-1)s}) \prod\limits_{\substack{p^r||k\\p^s\nmid n}} (-p^{(r-1)s}).
\end{align*}
On the other hand,
\begin{align*}
 (k^*)^s c_{\overline{k}}^{(s)}(n) &= \prod \limits_{\substack{p^r||k}}p^{(r-1)s} \prod\limits_{\substack{p|k\\p^s |  n}} c_{p}^{(s)}(n) \prod\limits_{\substack{p|k\\p^s \nmid  n}} c_{p}^{(s)}(n)\\
  &=\prod \limits_{\substack{p^r||k\\p^s |n}}p^{(r-1)s} \prod\limits_{\substack{p|k\\p^s |  n}} (p^s-1)
  \prod \limits_{\substack{p^r||k\\p^s\nmid n}}p^{(r-1)s} \prod\limits_{\substack{p|k\\p^s \nmid  n}} (-1)\\
  &=c_k^{(s)}(n(k^*)^s).
\end{align*}

 \end{proof}

 Now we give a sufficient condition for the existence of Cohen-Ramanujan expansions of certain arithmetical functions in the second variable. We call the Cohen-Ramanujan expansion in the first variable as the CR-series expansion and the expansion in the second variable as the CR'-series expansion. This result is analogous to \cite[Theorem 4]{johnson1982reciprocity}.

\begin{theo}
 Let $f$ be an arithmetical function. Let $a(k)$ be a function such that $a(k)=0$ if $k\neq (\overline{k})^s$. For other values of $k$, $a(k)$ is defined recursively using the relation
 \begin{align}\label{eq:fk-defn}
  f(k) = \sum\limits_{d|k} d^s\mu(d)a(d^s).
 \end{align}
Suppose \begin{align}
           \sum\limits_{k=1}^{\infty} \sum\limits_{n=1}^{\infty} |a(n^sk^s)| <\infty.
          \end{align}
Define
\begin{align}\label{eq:bk-defn}
b(k^s)=\sum\limits_{n=1}^{\infty} \mu(n)a(n^sk^s),
\end{align}
with $b$ defined to be 0 on positive integers that are not $s^{\text{th}}$ powers. Then
\begin{align}\label{eq:muk-fk}
 \mu(k)f(k) = \sum\limits_{n=1}^{\infty} b(n^s)c_k^{(s)}(n^s)
\end{align}
for all square-free $k$. Also, the right-hand side series in equation (\ref{eq:muk-fk}) converges absolutely.
\end{theo}
\begin{proof}
By the inversion formula we stated previously, we have $b(k^s)=\sum\limits_{n=1}^{\infty} \mu(n)a(n^sk^s)$ if and only if $a(k^s) = \sum\limits_{n=1}^{\infty}b(n^sk^s)$.  So from equation (\ref{eq:fk-defn}), we get
\begin{align*}
f(k) &= \sum\limits_{d|k} d^s\mu(d)a(d^s)\\
 &= \sum\limits_{d|k} d^s \mu(d) \sum\limits_{n=1}^{\infty}b(n^s d^s)\\
 &= \sum\limits_{d|k} d^s \mu(d) \sum\limits_{\substack{n=1\\d|n}}^{\infty}b(n^s)\\
 &= \sum\limits_{\substack{n=1}}^{\infty}b(n^s) \sum\limits_{d^s|(k^s,n^s)_s} d^s \mu(d)
\end{align*}
which is equal to $ \sum\limits_{\substack{n=1}}^{\infty}b(n^s) c_k^{(s)}(n^s) \mu(k)$ by Theorem \ref{th:mu-ck-sum}.

Now $\sum\limits_{\substack{n=1}}^{\infty}|b(n^s)|\, | c_k^{(s)}(n^s)| \leq k \sum\limits_{\substack{n=1}}^{\infty} \sum\limits_{\substack{m=1}}^{\infty}|a(n^sm^s)|$ implying the absolute convergence.

\end{proof}

Now we  give a converse of the above theorem.

\begin{theo}
 Suppose that $f$ has a series representation as in (\ref{eq:muk-fk}) with $\sum\limits_{\substack{k=1}}^{\infty}b(k^s)<\infty$ and $b(n)=0$ if $n\neq (\overline{n})^s$. Then $b(k^s)$ is given by (\ref{eq:bk-defn}).
\end{theo}
\begin{proof}
 We will more or less work backward the proof of the previous theorem. For square-free $k$, from equation (\ref{eq:muk-fk}), we have
 \begin{align*}
  f(k) &= \sum\limits_{n=1}^{\infty} b(n^s)c_k^{(s)}(n^s)\mu(k)\\
  &=  \sum\limits_{\substack{n=1}}^{\infty}b(n^s) \sum\limits_{d^s|(k^s,n^s)_s} d^s \mu(d)\\
   &= \sum\limits_{d|k} d^s \mu(d) \sum\limits_{\substack{n=1\\d|n}}^{\infty}b(n^s)\\
 &= \sum\limits_{d|k} d^s \mu(d) \sum\limits_{n=1}^{\infty}b(n^s d^s).
 \end{align*}
Thus $a(d^s) = \sum\limits_{n=1}^{\infty}b(n^s d^s)$ which implies that $b(d^s)=\sum\limits_{n=1}^{\infty} \mu(n) a(n^s d^s)$.
\end{proof}

\begin{lemm}\label{th:rev-k-n-in-ck}
 Let $k$ and $n$  be square-free integers. Then $\mu(k) c_k^{(s)}(n^s) = \mu(n) c_n^{(s)}(k^s)$.
\end{lemm}
\begin{proof}
 Note that  from Theorem \ref{th:mu-ck-sum}, we get an expression for $\mu(k) c_k^{(s)}(n^s)$  which is symmetric in $k$ and $n$ thus permitting us interchanging the positions of $k$ and $n$.
\end{proof}

Now we prove a \emph{reciprocity theorem}.
\begin{theo}\label{th:mu-c-reciprocity}
 For square-free integers $k$ and $n$,  we have
 \begin{align}
  \frac{\mu(\overline{k})}{(k^*)^s} c_k^{(s)}(n^s(k^*)^s) = \frac{\mu(\overline{n})}{(n^*)^s} c_n^{(s)}(k^s(n^*)^s).
 \end{align}
\end{theo}

\begin{proof}
We have
 \begin{align*}
  \mu(\overline{k}) c_{\overline{k}}^{(s)}(\overline{n}^s) &= \mu(\overline{k}) \sum\limits_{d^s|(\overline{k}^s, \overline{n}^s)}d^s\mu(d)\\
  &= \mu(\overline{k}) \sum\limits_{d^s|(\overline{k}^s, n^s)}d^s\mu(d)\\
  &= \mu(\overline{k}) c_{\overline{k}}^{(s)}(n^s)\\
  &= \mu(\overline{k}) \frac{ c_k^{(s)}(n^s (k^*)^s)}{(k^*)^s }
 \end{align*}
 where we used Lemma \ref{th:ck-ck-core} to arrive at the last step above. On the other hand, using Lemma \ref{th:rev-k-n-in-ck}, we get
 \begin{align*}
  \mu(\overline{k}) c_{\overline{k}}^{(s)}(\overline{n}^s) &=\mu(\overline{n}) c_{\overline{n}}^{(s)}(\overline{k}^s)\\
  &=  \mu(\overline{n}) c_{\overline{n}}^{(s)}(k^s)\\
  &= \mu(\overline{n}) \frac{ c_n^{(s)}(k^s (n^*)^s)}{(n^*)^s }.
 \end{align*}
\end{proof}
We are going to show now that the existence of a $CR$-series expansion is equivalent to the existence of $CR'$-series expansion for certain functions.
\begin{theo}\label{th:CR-exp-in-other-var}
 A function $f$ has an absolutely convergent Cohen-Ramanujan series expansion in the first variable $f(n) = \sum\limits_{k=1}^{\infty} a(k)c_k^{(s)}(n^s)$ where $a(k)=0$ if $k$ is not square-free if and only if $f$ has an absolutely convergent series expansion in the second variable  in the form $f(n)=\frac{\mu(\overline{n})}{(n^*)^s}\sum\limits_{k=1}^{\infty}\xi_{n^*}(k)\,b(k/n^*) c_n^{(s)}(k^s)$.
\end{theo}

\begin{proof}
 We mainly use the reciprocity theorem Theorem \ref{th:mu-c-reciprocity} in the proof. Suppose that $f$ has a  absolutely convergent $CR$-series expansion with coefficients $a(k)$ as given in the statement. Then
 \begin{align*}
f(n) &= \sum\limits_{k=1}^{\infty} a(k)c_k^{(s)}(n^s)\\
&=  \sum\limits_{k=1}^{\infty} \mu(k) a(k) \mu(k) c_k^{(s)}(n^s)\\
&=  \sum\limits_{k=1}^{\infty} \mu(k) a(k) \mu(\overline{k}) c_{\overline{k}}^{(s)}(n^s)\quad\text{since $\overline{k}=k$ when $k$ is square-free}\\
&=  \sum\limits_{k=1}^{\infty} \mu(k) a(k) \mu(\overline{k}) \frac{c_{k}^{(s)}(n^s (k^*)^s)}{(k^*)^s}\quad\text{by Lemma \ref{th:ck-ck-core}}.\\
 \end{align*}
 Now by the reciprocity theorem, the above is equal to 
 \begin{align*}
    \sum\limits_{k=1}^{\infty} \mu(k) a(k) \frac{\mu(\overline{n})}{(n^*)^s} c_{n}^{(s)}(k^s (n^*)^s) &= \frac{\mu(\overline{n})}{(n^*)^s} \sum\limits_{\substack{k=1\\n^*|k}}^{\infty} \mu(k/n^*) a(k/n^*)  c_{n}^{(s)}(k^s)\\
&= \frac{\mu(\overline{n})}{(n^*)^s}\sum\limits_{k=1}^{\infty}\xi_{n^*}(k)\,b(k/n^*) c_n^{(s)}(k^s)
 \end{align*}
where $b(k/n^*)=\mu(k/n^*) a(k/n^*) $ which is 0 if $k/n^*$ is not square-free.

Conversely suppose that $f(n^s) = \frac{\mu(\overline{n})}{(n^*)^s}\sum\limits_{k=1}^{\infty}\xi_{n^*}(k)\,b(k/n^*) c_n^{(s)}(k^s)$ where $b(m)=0$ if $m$ is not square-free.  Then
\begin{align*}
 f(n)  &= \sum\limits_{k=1}^{\infty} \frac{\mu(\overline{n})}{(n^*)^s}b(k) c_n^{(s)}(k^s (n^*)^s)\\
 &= \sum\limits_{k=1}^{\infty} b(k) \frac{\mu(\overline{k})}{(k^*)^s} c_k^{(s)}(n^s (k^*)^s) \quad\text{by Theorem \ref{th:mu-c-reciprocity}}\\
 &= \sum\limits_{k=1}^{\infty} b(k) \frac{\mu(\overline{k})}{(k^*)^s} \times (k^*)^s c_{\overline{k}}^{(s)}(n^s) \quad\text{by Lemma \ref{th:ck-ck-core}}\\
\end{align*}

Write $\alpha(k) = b(k) \frac{\mu(\overline{k})}{(k^*)^s}$. For a fixed positive integer $k$, note that there can be many positive integers $m$ such that $\overline{m}=k=\overline{k}$. Denote $H(k)=\{m\in \N : \overline{m}=k\}$. Then
\begin{align*}
 f(n^s) = \sum\limits_{k=1}^{\infty}\left(\sum\limits_{m\in H(k)}\alpha(m) (m^*)^s\right)c_{k}^{(s)}(n^s)=\sum\limits_{k=1}^{\infty}a(k)c_{k}^{(s)}(n^s)
\end{align*}
where $a(k) = \sum\limits_{m\in H(k)}\alpha(m) (m^*)^s$ for each square-free positive integer $k$.
\end{proof}

\begin{exam}
The construction of this example follows the methods used by Ramanujan in Sections 2, 3 and 9 in \cite{ramanujan1918certain}.
 Consider the Klee's function defined by $\Phi_s(n)=n\prod\limits_{p^s|n}\left(1-\frac{1}{p^s}\right)$.  $\Phi_s$ can be expanded as 
 \begin{align*}
  \Phi_s(n)  &= n\prod\limits_{p^s|n}\left(1+\frac{\mu(p)}{p^s}\right)\\
   &= n\sum\limits_{d^s|n}\frac{\mu(d)}{d^s}.
 \end{align*}
Let $\xi^{(s)}$ be the function such that $\xi^{(s)}_d (n)=d^s$ if $d^s|n$ and 0 otherwise. Hence equation (\ref{eq:crs-mu}) can be rewritten as $c_k^{(s)}(n) =  \sum\limits_{d|k}\xi^{(s)}_d (n) \frac{\mu(d)}{d^s}$ which by the M\"{o}bius inversion formula implies that 
$\xi^{(s)}_k (n) = \sum\limits_{d|k}c_d^{(s)}(n)$. Hence for $t\geq n$, we have
 \begin{align*}
  \Phi_s(n) &=\sum\limits_{d=1}^t \frac{1}{d^s} \xi^{(s)}_d (n) \mu(d)\frac{n}{d^s}\\
  &=\sum\limits_{d=1}^t \frac{1}{d^s} \sum\limits_{v|d}c_v^{(s)}(n) \mu(d)\frac{n}{d^s}.
   \end{align*}
Let us collect the coefficients of $c_1^{(s)}(n), c_2^{(s)}(n), \ldots$. Since $1|d$ for every $d$, the coefficient of  $c_1^{(s)}(n)$ is $\sum\limits_{d=1}^t \frac{1}{d^s} \mu(d)\frac{n}{d^s}$. Now $2|d$ for $d=2,4,\ldots,2[t/2]$ where for any real number $x$, $[x]$ represents the greatest integer not exceeding $x$. Hence instead of $d$, we may take $2d$ for $d$ in the range $1,2,\ldots, [t/2]$ in the coefficient so that the coefficient of $ c_2^{(s)}(n)$ will be  $\sum\limits_{d=1}^{[t/2]} \frac{1}{(2d)^s} \mu(2d)\frac{n}{(2d)^s}$. Continuing like this, we get 
\begin{align*}
  \Phi_s(n) &= c_1^{(s)}(n) \sum\limits_{d=1}^t \mu(d)\frac{n}{d^{2s}} +  \ldots + c_k^{(s)}(n)\sum\limits_{d=1}^{[t/k]}  \mu(kd)\frac{n}{(kd)^{2s}} +\ldots.
 \end{align*}
 Letting $t \rightarrow \infty$ and noting that $\sum\limits_{k=1}^{\infty}\frac{\mu(kd)}{d^{2s}} = \frac{\mu(k)k^{2s}}{J_{2s}(k)\zeta(2s)}$ (\cite[Equation 9.5]{ramanujan1918certain}), we get
\begin{align*}
  \frac{\Phi_s(n)\zeta(2s)}{n} &= \frac{\mu(1)}{J_{2s}(1)} c_1^{(s)}(n)  +  \ldots + \frac{\mu(k)}{J_{2s}(k)}c_k^{(s)}(n)  +\ldots.
 \end{align*} 
If we use the relation $\Phi_s(n^s)=J_s(n)$ between the Jordan totient function and Klee's function, then we get
\begin{align}
  \frac{\Phi_s(n)\zeta(2s)}{n} &= \frac{\mu(1)}{\Phi_{2s}(1)} c_1^{(s)}(n)  +  \ldots + \frac{\mu(k)}{\Phi_{2s}(k^s)}c_k^{(s)}(n)  +\ldots\label{eq:klee-series},
 \end{align}
 or, equivalently
 \begin{align*}
  \frac{J_s(n^s)\zeta(2s)}{n} &= \frac{\mu(1)}{J_{2s}(1)} c_1^{(s)}(n)  +  \ldots + \frac{\mu(k)}{J_{2s}(k)}c_k^{(s)}(n)  +\ldots.
 \end{align*} 
 It is interesting to compare the above with a similar expression equation 9.6 obtained by Ramanujan in \cite{ramanujan1918certain}. To avoid any confusion, note that Ramanujan used the symbol $\phi_s$ to denote the Jordan totient function in \cite{ramanujan1918certain}. Now the series expression  equation (\ref{eq:klee-series}) satisfies the requirement in Theorem \ref{th:CR-exp-in-other-var} with $a(k)=\frac{\mu(k)}{\Phi_{2s}(k^s)}$. Hence the $CR'$-series expansion for $ \frac{\Phi_s(n)\zeta(2s)}{n}$ will be 
 \begin{align*}
 \frac{\Phi_s(n)\zeta(2s)}{n} &=  \frac{\mu(\overline{n})}{(n^*)^s}\sum\limits_{k=1}^{\infty}\xi_{n^*}(k)\,\mu(k/n^*) a(k/n^*) c_n^{(s)}(k^s)\\
 &=\frac{\mu(\overline{n})}{(n^*)^s}\sum\limits_{k=1}^{\infty}\frac{\xi_{n^*}(k)}{\Phi_{2s}(k^s)}\, c_n^{(s)}(k^s).
  \end{align*}

\end{exam}

%\bibliographystyle{plain}
%\bibliography{references} 

\end{document}